\documentclass[12pt]{amsart}
\usepackage[utf8]{inputenc}
\usepackage{a4wide}
\usepackage{euscript}
\usepackage{fullpage}
\usepackage[T2A]{fontenc}
\usepackage{amsfonts}
\usepackage{amssymb, amsthm}
\usepackage{amsmath}
\usepackage{mathtools}
\usepackage{graphicx}
\usepackage{geometry}
\usepackage{tikz}
\usepackage{bbm}

\usepackage[colorlinks,link color=blue, cite color=blue,pagebackref,pdftex]{hyperref}

\renewcommand{\epsilon}{\varepsilon} 
\renewcommand{\phi}{\varphi} 

\def\geq{\geqslant}

\def\leq{\leqslant}
\def\le{\leqslant}

\def\kratno{\lower.5ex\hbox{$\,\vdots\,$}}

\def\q#1.{\smallbreak\noindent\hskip15pt{\bf#1.}\enspace\ignorespaces} 
\def\dotline{\smallskip\hbox to \hsize{\dotfill}\medskip}

\def\norm[#1]{\| #1 \|}
\newcommand{\spec}{\mathrm{Spec}}

\newcommand{\cov}{\mathrm{cov}}
\newcommand{\Vol}{\mathrm{Vol}}

\newcommand{\N}{\mathbb{N}}
\newcommand{\R}{\mathbb{R}}

\newcommand{\conv}{\mathop{\mathrm{conv}}\nolimits}

\theoremstyle{plain}
\newtheorem{thm}{Theorem} 
\newtheorem{lm}{Lemma}
\newtheorem{st}{Statement}

\theoremstyle{definition}

\newtheorem{cor}{Corollary}

\theoremstyle{remark}
\newtheorem{rem}{Remark}

\title{Mixed volume of infinite-dimensional convex compact sets}

\author[M. Dospolova]{Mariia Dospolova}
\address{Mariia Dospolova, St. Petersburg Department of Steklov Institute of~Mathematics, Russia} 
\email{dospolova.maria@yandex.ru}

\begin{document}

 \thanks{The work was supported by Ministry of Science and Higher Education of the Russian Federation, agreement № 075-15-2022-289.}

	
 	\begin{abstract}
Let $K$ be a convex compact $GB$-subset of a separable Hilbert space~$H$. Denote by $\spec_k K$ the set $\{(\xi_1(h), \ldots, \xi_k(h))\colon h\in K\}\subset \R^k,$ where $\xi_1, \ldots, \xi_k$ are independent copies of the isonormal Gaussian process on $H$. Tsirelson showed that in this case the intrinsic volumes of $K$ satisfy the relation
\begin{equation*}
V_k(K)= \frac{(2\pi)^{k/2}}{k!\kappa_k} \mathbf{E}\,\Vol_k(\spec_k K).
\end{equation*}
Here, $\mathbf{E} \ \Vol_k(\spec_k K)$ is the mean volume of $\spec_k K$ and $\kappa_k$ is the volume of the $k$-dimensional unit ball.

In this work, we generalize Tsirelson's theorem to the \textit{mixed volumes} of the infinite-dimensional convex compact $GB$-subsets of $H$, first introducing this notion.

Moreover, using the obtained result we compute the mixed volume of the closed convex hulls of the two orthogonal Wiener spirals.

 	\end{abstract}
 	\maketitle
 	\newpage
	
	\tableofcontents
	
	
	\newpage
	
	
	 
\section{Introduction}
\subsection{Intrinsic volumes}
Let $K \subset \R^d$ be a non-empty convex compact set and $\dim K$ be the dimension of $K$ (that is, the dimension of the smallest affine subspace containing $K$). One of the most important geometric characteristics of $K$ 
are its \textit{intrinsic volumes} $V_0(K), \ldots, V_d(K)$, which are defined as the coefficients in the Steiner formula (see, e.g., \cite[relation 14.5]{SW08})
	\begin{equation}\label{eq:steiner}
		\Vol_d(K+\lambda B^d) =\sum_{k=0}^d \kappa_{d-k} V_k(K) \lambda^{d-k}, \quad \lambda \geq 0,
	\end{equation}
	where 
	$\Vol_d (\cdot)$ denotes the volume ($d$-dimensional Lebesgue measure),
$B^k$ is the $k$-dimensional unit ball and $\kappa_k:=\Vol_k(B^k)=\pi^{k/2}/\Gamma(\frac k 2 +1)$ is the volume of $B^k$. 
In other words, the volume of the neighborhood 
is represented by a polynomial whose coefficients depend on the set $K$. 
	
	The intrinsic volumes play an important role in convex geometry (see, e.g., \cite{schneider2014convex}).
In particular, it can be shown \cite[Section 6.2]{SW08} that 
	$V_d(\cdot)$ is the $d$-dimensional volume, $V_{d-1}(\cdot)$ is half the surface area for $d$-dimensional convex compact sets, $V_1(\cdot)$ is the mean width, up to a constant factor, and $V_0(\cdot) \equiv 1$.
	
	Moreover, the normalization in  \eqref{eq:steiner} is chosen so that the intrinsic volumes of the set do not depend on the dimension of the ambient space. 
	This means that if we embed $K$ into $\R^N$ with $N\geq d$, the intrinsic volumes will be the same. This observation allowed Sudakov \cite{sudakov1979geometric} and Chevet \cite{chevet1976processus} to generalize the concept of intrinsic volume to the case of infinite-dimensional $K$ as follows.
	
	Let $H$ be an infinite-dimensional separable Hilbert space. Then for an arbitrary non-empty convex set $K \subset H$ we define ${V}_k(K), k=0, 1, \ldots$ by the formula
	\begin{align} \label{intrinf}
    {V}_k(K)  =   \sup_{K' \subset K
    } {V}_k(K') \in [0, \infty],
\end{align}
where the supremum is taken over all finite-dimensional convex compact subsets $K'$ of $K$.	

In the next subsection, we formulate the results demonstrating a deep connection between the intrinsic volumes of some convex compact sets and Gaussian processes.

\subsection{Sudakov’s and Tsirelson’s theorems}
A mean-zero Gaussian random process $(\xi(h))_{h \in H}$ over a separable Hilbert space $H$ is called \textit{isonormal} if its covariance function has the form
\begin{align*}
     \cov (\xi(h), \xi(g)) = \langle h, g\rangle,
\end{align*}
where $\langle \text{ },\text{ } \rangle $  denotes the inner product on $H$.

In his paper \cite[Proposition 14]{sudakov1979geometric} Sudakov discovered a connection between the first intrinsic volume and the expectation of the supremum of an isonormal process.
\begin{thm}[Sudakov]\label{theo:sudakov}
For a convex compact set $K\subset H$
\begin{equation}\label{2041}
V_1(K)= \sqrt{2\pi}\,\mathbf{E}\,\sup_{h\in K} \xi(h).
\end{equation}

\end{thm}

Later Tsirelson~\cite[Theorem 6]{Ts1985} generalized  Theorem \ref{theo:sudakov}.
Let $\{\xi_i(h)\colon h\in H\}$, $1\leq i \leq k$, denote $k$ independent copies of the isonormal process. Then $k$-\textit{dimensional spectrum} of a convex compact set $K \subset H$ is defined as the following random set:
$$
\spec_k K := \{(\xi_1(h), \ldots, \xi_k(h))\colon h\in K\}\subset \R^k.
$$

To formulate Tsirelson's result, we first need the notion of a $GB$-set.
A subset $K$ of a separable Hilbert space $H$ is said to be a $GB$-\textit{set} if there exists a modification of the isonormal process with 
index set $K$, which has almost surely
bounded realizations (see Section \ref{history} for detailed definitions and properties). It is known 
\cite[Theorem 1]{sudakov1979geometric} that the property of a convex $K$ to be a $GB$-set is equivalent to $V_1(K)<\infty$. In the latter case $V_k(K)<\infty$ for all $k=0, 1, \ldots$ (see, e.g., \cite{chevet1976processus}). 

\begin{thm}[Tsirelson]\label{theo:tsirelson_spectrum}
For all convex compact $GB$-sets $K\subset H$ and all $k=0,1, \ldots,$
\begin{equation}\label{2042}
V_k(K)= \frac{(2\pi)^{k/2}}{k!\kappa_k} \mathbf{E}\,\Vol_k(\spec_k K).
\end{equation}
\end{thm}

\begin{rem}\label{finited}
In the case when $K \subset \R^d$, $k \leq d$, the last formula can be rewritten as
\begin{equation*}
V_k(K)= \frac{(2\pi)^{k/2}}{k!\kappa_k} \mathbf{E}\,\Vol_k(AK),
\end{equation*}
where $A$ is a standard Gaussian matrix of size $k\times d$ (whose entries are independent standard normal random variables), $\spec_k K = AK:=\{Ax:x\in K\}\subset \R^k$.
\end{rem}

\begin{rem}
Strictly speaking,
in Theorem \ref{theo:sudakov} we need the existence of a \textit{separable} modification of the process $\xi$, and in Theorem \ref{theo:tsirelson_spectrum} we need the existence of a 
so-called \textit{natural} modification of $\xi_i$ (see Subsections \ref{snmod}, \ref{gbset} for details). We will show that under the assumptions of Theorems  \ref{theo:sudakov}, \ref{theo:tsirelson_spectrum} the corresponding modifications do exist (see Statement \ref{st2}).
\end{rem}
The main goal of this paper is to obtain a generalization of Theorem \ref{theo:tsirelson_spectrum} to the  \textit{mixed volumes} defined in the next subsection.

\subsection{Mixed volumes}
Minkowski proved \cite{minkowski1911theorie} 
that for arbitrary non-empty convex compact sets
$K_1, \ldots, K_s \subset \R^d$ \ the functional $\Vol_d(\lambda_1K_1+\ldots+\lambda_sK_s)$ for ${\lambda_1, \ldots, \lambda_s\geq 0}$ is a homogeneous polynomial of degree $d$ with the non-negative coefficients:
\begin{align}\label{mixed1}
    	\Vol_d(\lambda_1K_1+\ldots+\lambda_sK_s)=\sum_{i_1=1}^s \cdots \sum_{i_d=1}^s \lambda_{i_1} \ldots \lambda_{i_d}  \Tilde{V}_d(K_{i_1}, \ldots, K_{i_d}).
\end{align}

The coefficients $\Tilde{V}_d(K_{i_1}, \ldots, K_{i_d})$ are uniquely determined if we assume that they are symmetric with respect to the permutations of $K_{i_1}, \ldots, K_{i_d}$. The coefficient $\Tilde{V}_d(K_{i_1}, \ldots, K_{i_d})$ is called the \textit{mixed volume} of $K_{i_1}, \ldots, K_{i_d}$. 

It is easy to understand (see, e.g., \cite[Section 5.1]{schneider2014convex}) that intrinsic volumes are special cases of the mixed volumes, namely, 

\begin{align}\label{mixed}
    	V_k(K)= \frac{{d \choose k}}{\kappa_{d-k}}\Tilde{V}_d(\underbrace{K,\dots,K}_{k\;\text{times}},B^d,\dots,B^d) .
\end{align}

The theory of mixed volumes finds wide application in convex and algebraic geometry \cite[Chapter 4]{burago2013geometric}, inequalities \cite{schneider2014convex} and 
the theory of Gaussian distributions \cite{zaporozhets2014random}.
Some of the properties of the mixed volumes are given in  Subsection \ref{mixprop}.

Next, we formulate the main results of this work.



 \section{Main results}
  \subsection{Generalization of Tsirelson's theorem}
In order to generalize Theorem \ref{theo:tsirelson_spectrum} to the case of mixed volumes, we first define an isonormal Gaussian random process according to Tsirelson \cite{Ts1982}.

Consider 
a linear topological space with  mean-zero Gaussian measure $(E, \gamma)$ and its kernel $E_0 \subset E$ (see Section \ref{history} for definitions and properties). Since the kernel is a Hilbert space, we have the inner product on $E_0$, which we will denote by $ \langle \text{ },\text{ } \rangle_{E_0} $ (it is uniquely determined by the measure $\gamma$). For each $\theta \in E_0$ the linear functional $ \langle \theta,\eta \rangle_{E_0} $  is continuous in $\eta \in E_0$ and has a unique (up to coincidence almost everywhere) extension to a linear functional, measurable in $x \in E$ (see \cite[Section 9, Lemma 2]{ML1995} or \cite[Corollary 2.10.8]{bogachev1998gaussian}), which we
denote by $\langle \theta,x \rangle $. Moreover,
\begin{align} \label{isometr}
\int_{E} \langle \theta,x \rangle^2 \gamma(dx) = \norm[\theta]^2 =  \langle \theta,\theta \rangle.
\end{align}

Thus, for any set $K \subset E_0$, the isonormal Gaussian random process $\langle \theta,x \rangle $ is defined, where $\theta \in K$, $x$ ranges over the space $E$ equipped with the Gaussian measure $\gamma$.

To state and prove the main result, we will use the kernel $E_0$ as $H$ and the process $\langle \theta, \cdot \rangle $ as isonormal process.

  Let us rewrite Theorem \ref{theo:tsirelson_spectrum} according to the notation of this subsection for further convenience. Formula \eqref{2042} for $k = 0, 1, \ldots$ turns to
\begin{align*}
        V_k(K)=  \frac{(2\pi)^{k/2}}{k!\kappa_k} \int_{E} \int_{E}\ldots \int_{E} \Vol_k(\spec(x_1, \ldots, x_k | K)) \gamma(dx_1) \ldots \gamma(dx_k).
    \end{align*}
Here $K \subset E_0$ is a convex compact $GB$-set, \\ ${ \spec(x_1, \ldots, x_k | K) := \{ \left( \langle \theta,x_1 \rangle, \ldots, \langle \theta,x_k \rangle \right) : \theta \in K \} \subset \R^k}$ is the joint spectrum for $x_1, \ldots, x_k \in E$ on $K$.

Now we introduce the concept of mixed volume for infinite-dimensional convex sets 
similar to \eqref{intrinf}.

Let $K_1, \ldots, K_k \subset H$ be non-empty convex 
subsets of an infinite-dimensional separable Hilbert space $H$. 
Then the \textit{mixed volume} $\Tilde{V}(K_{1}, \ldots, K_{k})$ of the sets $K_1, \ldots, K_k$ is defined as
\begin{align}\label{commonmix}
  \Tilde{V}(K_{1}, \ldots, K_{k})  =   \sup_{K'_i \subset K_i}
  \frac{{d \choose k}}{\kappa_{d-k}} \Tilde{V}_d(K'_{1}, \ldots, K'_{k}, \underbrace{B^d,\dots,B^d}_{d-k\;\text{times}}),
\end{align}
where the supremum is taken over all $d \geq k$ and all finite-dimensional convex compact subsets $K'_i \subset K_i, \dim K'_i \leq d, 
  i = 1, \ldots, k$.
\begin{rem}\label{norm}
The normalization in \eqref{commonmix} is chosen so that for  $K_i: \dim K_i \leq d$, the right-hand
side  
of \eqref{commonmix} does not depend on $d$, as well as in expression 
\eqref{mixed}. Therefore, $ \Tilde{V}(K_{1}, \ldots, K_{k})$ is well defined.

\end{rem}
The proof of Remark \ref{norm} can be found in Subsection \ref{normproof}. 


Now we are ready to formulate the main result of this paper. 

\begin{thm}\label{mainmix} Fix $k \in \N$.
For convex compact $GB$-sets $K_i \subset E_0, \ i = 1, \ldots, k$  we have 
\begin{gather*} 
    \Tilde{V}(K_{1}, \ldots, K_{k}) =
    \frac{(2\pi)^{k/2}}{k!\kappa_k}  \ \mathbf{E} \ \Tilde{V}_k(\spec_k K_1, \ldots, \spec_kK_k)
\\
=\frac{(2\pi)^{k/2}}{k!\kappa_k} \int_{E} \ldots \int_{E} \Tilde{V}_k(\spec(x_1, \ldots, x_k | K_1), \ldots, \spec(x_1, \ldots, x_k | K_k)) \gamma(dx_1) \ldots \gamma(dx_k).
\end{gather*}
\end{thm}
\begin{rem}
$GB$-property of the sets $K_i$ ensures almost everywhere boundedness (and convexity) of the sets $\spec(x_1, \ldots, x_k | K_i)$. 
\end{rem}

\subsection{Example: mixed volume of the closed convex hulls of two orthogonal Wiener spirals}
Let us first recall the definition of the \textit{Wiener spiral} introduced by Kolmogorov \cite{kolm1985}. 
The set of functions
$$
\{\mathbbm {1}_{[0,t]}(\cdot) \colon t\in [0, 1]\}\subset L^2[0, 1]
\;\;\;
$$
is called the \textit{Wiener spiral}. This set is an important object in functional analysis \cite{kolm1985}.

Recall that the \textit{convex hull} of a set $F$ is the smallest convex set containing $F$.

Gao and Vitale \cite{gao2001intrinsic} calculated the intrinsic volumes of the closed convex hull $K$ of the Wiener spiral:
\begin{align}\label{1319}
	V_k(K)=\frac{\kappa_k}{k!} =
	\frac{\pi^{k/2}}{\Gamma\left(\frac{k}{2}+1\right) k!}.
\end{align}

This was probably the first result that gave an explicit formula for the intrinsic volumes of a non-trivial infinite-dimensional convex compact set. 
Later similar results were proved for other infinite-dimensional convex compact sets \cite{kz2016intrinsic}.

In particular, \eqref{1319} implies that $V_1(K) < \infty$, so $K$ is a $GB$-set (see 
Theorem \ref{Sudakov} in Subsection \ref{gbset}).

The Wiener spiral is closely related to the Wiener process. Let $\{W(t)\colon t\geq 0\}$ be the standard one-dimensional Brownian motion. Consider the standard two-dimensional Brownian motion
$$
\{X^{(2)}(t)=(W_1(t),W_2(t))\colon t\geq 0\},
$$
where $W_1(t), W_2(t)$ are independent copies of $W(t)$. It is easy to see that $\spec_2 K$ has the same distribution as the closed convex hull of the two-dimensional Brownian motion  $\{X^{(2)}(t) \colon t\in [0,1]\}$.

Consider two Wiener spirals $S_1$ and $S_2$ in $L^2 [0,2]$:
$$
S_1 = \{\mathbbm {1}_{[0,t]}(\cdot) \colon t\in [0, 1]\}\subset L^2[0, 2] \;\;\;
\text{ and }
\;\;\;
S_2 = \{\mathbbm {1}_{[1+t,2]}(\cdot) \colon t\in [0, 1]\}\subset L^2[0, 2].
$$
We denote the corresponding closed convex hulls by $K_1$ and $K_2$.
In our next theorem, we compute $\Tilde{V}(K_{1}, K_{2})$.

\begin{thm}\label{thspW}
	For the closed convex hulls $K_1$ and $K_2$ of two orthogonal Wiener spirals we have
$$\Tilde{V}(K_{1}, K_{2})=2. $$
\end{thm}
  The proof of Theorem \ref{thspW}  uses Theorem \ref{mainmix} (see Section \ref{example}).


Let us conclude this introductory part by describing how the rest of the paper is organized.
The next section contains the necessary concepts, definitions and
facts from the theory of random processes and convex geometry, which supplement the information presented in the first two sections. In particular, in Subsection \ref{gbset} we formulate and prove Statement \ref{st2} auxiliary to Theorem \ref{mainmix} about one of the interpretations of the $GB$-property of convex compact sets.
Sections \ref{normproof} and \ref{mainmixproof} contain proofs of Remark \ref{norm} and  main Theorem \ref{mainmix}, respectively. Finally, the proof of Theorem \ref{thspW} is presented  in  Section \ref{example}.


\section{Preliminaries}\label{history}
\subsection{Gaussian vectors in linear spaces}\label{1.1}
Following \cite[Сhapters 1, 4]{ML2016} and \cite[Sections 8, 9]{ML1995}, we present the definition and basic properties of a Gaussian vector in a linear space.

Let $E$ be a linear topological space, $E^*$ be the space of continuous linear functionals on $E$.
A \textit{random vector} $X$  taking
values in $E$ is defined as a measurable mapping from some probability space $(\Omega, \mathcal{B}, \mathbf{P})$ to $E$.
At the same time, it is assumed that the corresponding $\sigma$-algebra of the space $E$ is large enough: all continuous linear functionals on $E$ are measurable with respect to it.

A random vector $X \in E$ is called $\textit{Gaussian}$ if $f(X)$ is a normal random variable for all $f \in E^*$.

An element $a \in E$ is said to be the $\textit{expectation}$ of $ X$ if $\mathbf{E} f(X)= f(a)$ for all $f \in E^*$.
A linear operator $C: E^* \rightarrow E$ is called the \textit{covariance operator} of 
$X$ if for any $f_1, f_2 \in E^*$
\begin{align*}
    \cov(f_1(X), f_2(X)) = f_1(Cf_2),
\end{align*}
where $ \cov(\cdot, \cdot)$ denotes the covariance between two random variables.
  The covariance operator $C$ has the following properties:
	\begin{enumerate} 
		\item 
		   $ f(Cg) = g(Cf) \text{   } \
		    \forall f,g \in E^* \text{ (symmetry)}$;
		\item
		     $f(Cf) \geq 0 \text{   } \
		     \forall f\in E^* \text{ (non-negative definiteness)}$.
	\end{enumerate}
	

The definition of a Gaussian vector makes sense when the space of continuous linear functionals on $E$ is rich enough. To this end, we will tacitly assume everywhere below that $E$ is a locally convex linear topological space, and the distribution of $X$ is a Radon measure. In this case, any Gaussian vector $X$ has an expectation and a covariance operator \cite[Section 8]{ML1995} that uniquely determine the distribution of~$X$. Therefore, similarly to the finite-dimensional case, we denote by $N(a,C)$ the distribution of the Gaussian vector $X$ with expectation $a$ and covariance operator~$C$. Under the above assumptions, the distributions of all Gaussian vectors have 
form~$N(a,C)$.


	In the following, we will be interested in mean-zero case when $a =0$. 

\subsection{Measurable linear functionals and kernel}
Consider a Gaussian vector~$X$ taking values in the linear space $E$. We will assume that $a =0$.
Denote by ${\gamma = N(0,C)}$ the distribution of $X$ in $E$. 

By definition of a Gaussian vector, the random variable $f(X)$ has normal distribution, so
\begin{align*}
    \mathbf{E}f(X)^2 = \int_E |f(x)|^2 \gamma(dx) <\infty.
\end{align*}

Thus, a canonical embedding $I^*$ of the space $E^*$ into the Hilbert space $L_2(E, \gamma)$ is well defined. The closure of the image $I^*(E^*)$ in $L_2(E, \gamma)$ is said to be the space of \textit{measurable linear functionals} and denoted by $E^*_{\gamma}$.

 The inner product in $E^*_{\gamma}$ is inherited from $L_2(E, \gamma)$:
\begin{gather*}
 \langle g_1, g_2 \rangle_{E^*_{\gamma}} = \int_E g_1(x)g_2(x)\gamma(dx) = \mathbf{E}g_1(X)g_2(X);
     \\
     \norm[g]^2_{E^*_{\gamma}}=\mathbf{E}g(X)^2.   
\end{gather*}

    In what follows, we treat the operator $I^*$ as the embedding $ I^*:E^* \rightarrow E^*_{\gamma}$.
     We define the dual operator $I: E^*_{\gamma} \rightarrow E$ by the following relation:
    \begin{align*}
        f(Ig) = \langle I^*f, g \rangle_{E^*_{\gamma}} = \mathbf{E} f(X) g(X), \quad  \forall f \in  E^*, g \in E^*_{\gamma}.
    \end{align*}

 It is known 
 \cite[Section 4.1]{ML2016} that under the assumptions stated in Subsection \ref{1.1}, the dual operator $I$ exists, it is linear and injective, and, moreover, the covariance operator $C$ can be factorized as 
 \begin{align*}
     C = II^*.
 \end{align*}
 
 Finally, the \textit{kernel} is defined as the set $E_0 := I(E^*_{\gamma}) \subset E$ equipped with inner product
 \begin{align*}
     \langle \theta_1, \theta_2 \rangle_{E_0} := \langle I^{-1}\theta_1, I^{-1}\theta_2 \rangle_{E^*_{\gamma}}, \quad \theta_1, \theta_2 \in E_0,
 \end{align*}
   and hence with norm
   \begin{align*}
     \norm[\theta]^2:=\norm[\theta]^2_{E_0} = \langle \theta, \theta \rangle_{E_0}, \quad \theta \in E_0.
 \end{align*}
 
  The norm is well defined since the operator $I$ is injective.
 
 
 Thus, the kernel is uniquely determined by the measure $\gamma$ and provides the key information about it (see \cite{ML2016}).
 
  We collect some properties of the kernel (see \cite[Section 4.1]{ML2016}).
 	\begin{enumerate} 
		\item $C(E^*)\subset E_0 \subset E$.
		 If the kernel is finite-dimensional, then in the nondegenerate case these three spaces coincide, otherwise they are all distinct. 
		\item If $E_0$ is infinite-dimensional, then $\gamma(E_0)=0$.
		\item The space $E_0$ is separable.
		\item The balls $\{ \theta \in E_0 : \norm[\theta]
		\leq R\}, \ R >0,$ are compact sets in $E$.
	\end{enumerate}

\subsection{Separable and natural modifications of process} \label{snmod}
Let $(\Omega, \mathcal{B}, \mathbf{P})$ be a probability space  and $T$ be a metric space. A random process $\xi(t, \omega)$, $t \in T$, $\omega \in \Omega$, is said to be \textit{separable} if there exists at most countable set $S \subset T$ (a \textit{separant} of the process) 
such that for any open set $U \subset T$ with probability $1$ the following equalities hold:
   \begin{align*}
   	\sup_{t \in U} \xi (t) = 	\sup_{t \in U \cap S} \xi (t), \quad 	\inf_{t \in U} \xi (t) = 	\inf_{t \in U \cap S} \xi (t).
   	\end{align*}


The following theorem (see \cite[Proposition 2.6.5]{bogachev1998gaussian}) provides a sufficient condition for the existence of a separable modification of a mean-zero Gaussian process. Recall that a random process $(\eta(t))_{t\in T}$ is called a \textit{modification} of the process $(\xi(t))_{t \in T}$ if these processes are defined on the same probability space and $\mathbf{P}\left(\xi(t) = \eta(t) \right) =1 $ for any $t \in T$.
A \textit {realization} of the process $(\xi(t))_{t \in T}$ is the function $ t\mapsto \xi(t,\omega) $ for some fixed $\omega \in \Omega$.
\begin{thm}
\label{separ}
Consider a mean-zero Gaussian random process $\left(\xi(t)\right)_{t \in T}$ on a set~$T$. Suppose that $T$ with semimetric $d(t,s) = \sqrt{\mathbf{E}\ |\xi(t) -\xi(s)|^2}$ is separable.  
Then, on the same probability space, there exists a separable mean-zero Gaussian random process $(\eta(t))_{t\in T}$ on $T$ such that for any fixed $t \in T$ one has $ \xi(t )=\eta(t)$ almost surely.
\end{thm}
To prove Theorem \ref{mainmix}, the existence of a separable modification of the process $\langle \theta,x \rangle $  is not sufficient.  We need the so-called \textit{natural} modification introduced by Tsirelson~\cite{Ts1976}.

A modification $(\eta(t))_{t\in T}$ of the process $(\xi(t))_{t \in T}$ is called \textit{natural} if there exists a metric $\rho_1$ on $ T$ such that
$(T, \rho_1)$ is a separable metric space and the process $(\eta(t))_{t\in T}$ has almost surely
continuous realizations on $(T, \rho_1)$. 

Below we formulate a theorem (see \cite[Section 7]{ML1995} or \cite[Theorem 2.6.3, Proposition 2.6.4]{bogachev1998gaussian}) that allows us  to check the existence of a natural modification in terms of the \textit{oscillations} 
$\alpha$.
\begin{thm}
\label{Bogachev}
Let $(T, \rho)$ be a separable metric space and let $(\xi(t))_{t\in T}$ be a mean-zero separable Gaussian random process with the continuous covariance function
\begin{align*}
    (t,s) \mapsto \mathbf{E} \ \xi(t) \xi(s).
\end{align*}
Then there exists a non-random function $\alpha : T \rightarrow [0, \infty]$ such that with probability $1$ for all $t\in T$
\begin{align*}
    \alpha(t) = \lim_{\varepsilon \rightarrow 0} \sup \{ |\xi(u, \omega) - \xi(v, \omega)|, u, v \in B(t, \varepsilon)\},
\end{align*}
where $B(t, \varepsilon)$ denotes the open ball of radius $\varepsilon$ centered at $t$.

Moreover, if $\alpha(t)< \infty$ for all $t \in T$, then the process $(\xi(t))_{t\in T}$ has a natural modification.
\end{thm}

\subsection{$GB$-sets: equivalent definitions and properties} \label{gbset}
As mentioned in the introduction, the $GB$-set is a subset $K$ of a separable Hilbert space $H$ such that there exists a modification of the isonormal process with 
index set $K$, which has almost surely
bounded realizations.

In this subsection, we formulate the results of Sudakov \cite[Theorem 1]{sudakov1979geometric} and Tsirelson \cite[Theorem 3]{Ts1976} on equivalent definitions of the $GB$-set, and also prove an auxiliary Statement \ref{st2} about the connection between the $GB$-property of a set and the oscillation of a corresponding process.

\begin{thm}[Sudakov]\label{Sudakov}
Let $K \subset H$ be a 
convex 
subset of the Hilbert space $H$. The following statements are equivalent:
\begin{enumerate}
    \item the set $K$ is a $GB$-set; 
\item $V_1(K) < \infty$.
\end{enumerate}
\end{thm}
\begin{thm}[Tsirelson]\label{Tsirelson1}
Let $K \subset H$ be a subset of the Hilbert space $H$. The following statements are equivalent:
\begin{enumerate}
    \item  the 
    isonormal process on the set $K$ has a natural modification; 
\item  the set $K$ is $GB_\sigma$-set (that is a countable union of $GB$-sets).
\end{enumerate}
\end{thm}

\begin{st}\label{st2}
Let $K\subset E_0 $ be a convex compact $GB$-set. Then Theorem \ref{Bogachev} holds for $T = K$ with the standard metric generated by the inner product and for the process $\langle \theta,\cdot \rangle$ on $K$. Thus, the process $\langle \theta,\cdot \rangle$ on $K$ has a natural modification.

Moreover, the converse also holds. Consider a convex compact set $K\subset E_0 $ with standard metric satisfying all conditions of Theorem \ref{Bogachev}. Then $K$ is a $GB$-set, equivalently, 
$V_1(K) < \infty$. 
\end{st}
\begin{rem}
Note that Statement \ref{st2} can be deduced from Theorem \ref{Tsirelson1}. Nevertheless, we will provide an alternative proof of Statement \ref{st2} for the reader’s convenience.

\end{rem}
\begin{rem}
	The existence of a separable modification of the process $\langle \theta,\cdot \rangle$ does not require the $GB$-property of the compact set $K$, as can be seen from the proof below. 

\end{rem}
\begin{proof}[Proof of Statement \ref{st2}]


First, let us check whether the process $\langle \theta,x \rangle, \ \theta \in K$ has a separable modification.

We will use Theorem \ref{separ}.
Note that for $\theta_1, \theta_2 \in K$  we have
\begin{align*}
    d(\theta_1, \theta_2) := \sqrt{\mathbf{E}|\langle \theta_1,x \rangle - \langle \theta_2,x \rangle|^2} =
    \sqrt{\mathbf{E} \left( \langle \theta_1,x \rangle^2 + \langle \theta_2,x \rangle^2 - 2 \langle \theta_1,x \rangle \langle\theta_2,x \rangle \right)}
    \\
    =\sqrt{\langle \theta_1, \theta_1 \rangle + \langle \theta_2, \theta_2 \rangle - 2 \langle \theta_1, \theta_2 \rangle}  = \sqrt{\langle \theta_1 - \theta_2, \theta_1 - \theta_2 \rangle} = \norm[\theta_1 - \theta_2].
\end{align*}
Here in the third equality, we used the fact that the process  $\langle \theta,x \rangle, \ \theta \in K$ is isonormal. 
Therefore, the semimetric $d$ coincides with the standard metric on $K$, and hence $K$ is separable with this semimetric. Then, by Theorem \ref{separ}, we can assume without loss of generality that the process $\langle \theta, \cdot \rangle$ is separable.

Now we need to check the continuity of the covariance function. 

Let $(\theta_1, \theta_2) \in K\times K$ and $\norm[\theta_1^n -\theta_1] \rightarrow 0$, $\norm[\theta_2^n -\theta_2] \rightarrow 0$ as $n\rightarrow \infty$. Let us show that 
\begin{align*}
    \mathbf{E} \langle \theta_1^n ,x \rangle \langle \theta_2^n ,x \rangle 
    = 
    \int_{E} \langle \theta_1^n ,x \rangle \langle \theta_2^n ,x \rangle \gamma(dx) 
    \rightarrow
    \int_{E} \langle \theta_1 ,x \rangle \langle \theta_2 ,x \rangle \gamma(dx)
    = \mathbf{E} \langle \theta_1 ,x \rangle \langle \theta_2 ,x \rangle.
\end{align*}
Indeed, by the isonormality of the process and the Cauchy–Schwarz inequality,
\begin{gather*}
   \left|\int_{E} \langle \theta_1^n ,x \rangle \langle \theta_2^n ,x \rangle \gamma(dx) 
    -
    \int_{E} \langle \theta_1 ,x \rangle \langle \theta_2 ,x \rangle \gamma(dx) \right|
    \\ = \left|\langle \theta_1^n, \theta_2^n  \rangle - \langle \theta_1 , \theta_2 \rangle \right|
    \\ \leq
    \left|\langle \theta_1^n - \theta_1, \theta_2  \rangle \right| 
    +  \left| \langle \theta_1^n - \theta_1, \theta_2^n - \theta_2 \rangle \right|
    +  \left| \langle \theta_2^n - \theta_2, \theta_1\rangle \right|
    \\
    \leq \norm[\theta_1^n-\theta_1]\norm[\theta_2]+\norm[\theta_1^n-\theta_1]\norm[\theta_2^n-\theta_2]+ \norm[\theta_2^n-\theta_2]\norm[\theta_1].
\end{gather*}
Letting $n \rightarrow \infty$ in the last inequality leads to the required relation.

Finally, let us verify that the oscillation $\alpha$ introduced in Theorem  \ref{Bogachev} is finite in the case when $K$ is a convex compact $GB$-set.

We have $V_1(K) < \infty$ by Theorem \ref{Sudakov}.
We will need formula \eqref{2041} for $V_1(K)$:
\begin{align*}
    V_1(K) = \sqrt{2\pi} \int_E (\sup_{\theta \in K} \langle \theta ,x \rangle)\gamma(dx).
\end{align*}


Assume that there exist $\theta \in K, E_1 \subset E, \gamma(E_1) >0$ such that for $x \in E_1$
\begin{align*}
    \alpha(\theta) = \lim_{\varepsilon \rightarrow 0} \sup \{ | \langle \theta_1 ,x \rangle -  \langle \theta_2 ,x \rangle|, \theta_1, \theta_2 \in B(\theta, \varepsilon)\} = \infty.
\end{align*}
Since for $x \in E_1$
$$\infty = \alpha(\theta)  \leq 2 \sup_{\theta \in K} |\langle \theta ,x \rangle|,  $$
$\sup_{\theta \in K} - \langle \theta ,x \rangle = \sup_{\theta \in K}  \langle \theta , -x \rangle $, and the distribution $\gamma$ is symmetric, we get a contradiction with finiteness of $V_1(K)$. 
This means that $\alpha(\theta) < \infty$. 

Conversely, suppose that $\alpha(\theta) < \infty$ for all $\theta \in K$. Let us prove that in this case
\begin{align*}
    \gamma(x \in E: \sup_{\theta \in K} |\langle \theta ,x \rangle | < \infty)=1.
\end{align*}

We fix $x \in E_1$, where $E_1 \subset E$ is a set of full measure on which $ \alpha(\theta) < \infty$ for all~$\theta \in K$.

For each $\theta \in K$ we choose $\Tilde{\epsilon}(\theta), M(\theta)< \infty$ such that
\begin{align} \label{fin1}
    \sup \{ | \langle \theta_1 ,x \rangle -  \langle \theta_2 ,x \rangle|, \theta_1, \theta_2 \in B(\theta, \Tilde{\epsilon}(\theta))\} < M(\theta).
\end{align}

Consider the covering of $K$ by balls $\{B(\theta, \Tilde{\epsilon}(\theta))\}_{\theta \in K}$. Since $K$ is compact, we can choose a finite subcovering of $K$ of the form $\{B(\theta^i, \Tilde{\epsilon}(\theta^i))\}_{i=1}^{N } = \{B_i\}_{i=1}^{N}.$

Then by the linearity of the process $\langle \theta, \cdot \rangle $ and by relation \eqref{fin1}, we have
\begin{gather*}
	\sup_{\theta \in K} |\langle \theta ,x \rangle | 
	= \max_{1\leq i \leq N} \sup_{\theta \in  B_i} |\langle \theta-\theta^i,x \rangle + \langle \theta^i,x \rangle | 
	\\
	\leq \max_{1\leq i \leq N}  |\langle \theta^i,x \rangle|  +    \max_{1\leq i \leq N}\sup_{\theta\in B_i} |\langle \theta-\theta^i,x \rangle |  
	\\ 
	\leq  \max_{1\leq i \leq N}  |\langle \theta^i,x \rangle|  + \max_{1\leq i \leq N}M(\theta^i)
	< \infty
	.
\end{gather*}

Since the last inequality holds for all $x\in E_1$ by our assumption, and $\gamma (E_1) = 1$, we have $\gamma(x \in E: \sup_{ \theta \in K} |\langle \theta ,x \rangle| < \infty)=1.$

Thus, $K$ is a $GB$-set. Then by Theorem \ref{Sudakov}, we obtain that $V_1(K) < \infty.$

The statement is proved.
\end{proof}
\subsection{Properties of mixed volumes}\label{mixprop}

We collect the basic properties of the mixed volumes defined in the introduction, some of which we will need in the proof of Theorem \ref{mainmix}. For a more detailed introduction to mixed volume theory, we refer to \\ {\cite[Chapter 4]{burago2013geometric}}~and \cite[Chapter 5]{schneider2014convex}.
 
\begin{enumerate}
    \item $\Tilde{V}_d(K,\dots ,K) = \Vol_d(K).$
    \item Independence of order: \label{symm}
    $$\Tilde{V}_d(K_{1}, \ldots, K_{d}) = \Tilde{V}_d(K_{\sigma_1}, \ldots, K_{\sigma_d}),$$
    where $\sigma$ is an arbitrary permutation of numbers $1, \ldots, d$.
    \item Non-negative multilinearity:
    \begin{gather*}
    	\Tilde{V}_d(\lambda K_{1}+ \lambda' K'_{1}, K_2 \ldots, K_{d})
    	\\
    	 = \lambda\Tilde{V}_d( K_{1}, K_2 \ldots, K_{d})+  \lambda' \Tilde{V}_d( K'_{1}, K_2 \ldots, K_{d}) \text{ for } \lambda, \lambda' \geq 0.
    	\end{gather*}
    \item Invariance with respect to a parallel translation:
     \begin{align*}
    	\Tilde{V}_d( K_{1}+a_1,   \ldots, K_{d}+a_d) = 	\Tilde{V}_d( K_{1},  \ldots, K_{d})
    \end{align*}
     for any $a_1, \ldots, a_d \in \R^d$.
\item Invariance with respect to a unimodular affine transformation $O$:
\begin{align*}
	\Tilde{V}_d( OK_{1},   \ldots, OK_{d}) = 	\Tilde{V}_d( K_{1},  \ldots, K_{d}).
\end{align*}
    \item \label{monoton} Monotonicity with respect to each argument: let $K_i, L_i, \ i =1, \ldots, d,$ be  convex compact sets, and $K_i \subset L_i$. Then
     \begin{align*}
    	\Tilde{V}_d( K_{1},   \ldots, K_{d}) \leq 	\Tilde{V}_d( L_{1},  \ldots, L_{d}).
    \end{align*}
This property implies the non-negativity of the mixed volumes.
\item Additivity: if $A, B, A \cup B \subset \R^d$ are non-empty convex compact sets, then
	\begin{align*}
		\Tilde{V}_d(\underbrace{A \cup B,\dots, A \cup B}_{i\;\text{times}},K_{i+1},   \ldots, K_{d}) =		\Tilde{V}_d(A ,\dots, A
		,K_{i+1},   \ldots, K_{d})
		 \\ +  \	\Tilde{V}_d(B,\dots,  B
		,K_{i+1},   \ldots, K_{d}) - 	\Tilde{V}_d(A \cap B,\dots, A \cap B
		,K_{i+1},   \ldots, K_{d}).
	\end{align*}
\item  \label{contin} The mixed volumes are continuous with respect to each of the $K_i$ and to the family of these sets in the Hausdorff metric $d_H$ \\ ($d_H(K_1, K_2) :=\inf \{\varepsilon \geq 0\,:\ K_1\subset K_2 + \varepsilon B^d {\text{ and }}K_2\subset K_1+ \varepsilon B^d\}$).
\end{enumerate}



\section{Proof of Remark \ref{norm}} \label{normproof}

Let $K_i \subset \R^{d}$. It is sufficient to prove that
\begin{align*}
 \frac{{d \choose k}}{\kappa_{d-k}} \Tilde{V}_d(K_{1}, \ldots, K_{k}, B^d,\dots,B^d) =  \frac{{d+1 \choose k}}{\kappa_{d+1-k}} \Tilde{V}_{d+1}(K_{1}, \ldots, K_{k}, B^{d+1},\dots,B^{d+1}).
\end{align*}

Indeed, according to Minkowski's formula \eqref{mixed1},
\begin{align*}
    	\Vol_{d+1}(\lambda_1K_1+\ldots+\lambda_kK_k+\lambda B^{d+1})
    	\\
         = \int_{-\lambda}^{\lambda}	\Vol_{d}(\lambda_1K_1+\ldots+\lambda_kK_k+\sqrt{\lambda^2 - z^2}B^d)dz
         \\
       = \int_{-\lambda}^{\lambda}
       \sum_{i_1=1}^{k+1} \cdots \sum_{i_d=1}^{k+1} \lambda_{i_1} \ldots \lambda_{i_d}  \Tilde{V}_d(K_{i_1}, \ldots, K_{i_d})dz,
\end{align*}
where  $ \lambda_{i_j} \in \{  \lambda_1, \ldots, \lambda_k, \sqrt{\lambda^2 - z^2}\}, \ K_{i_j} \in \{ K_1, \ldots, K_k, B^d\}$.

Now let us look at the coefficient of the monomial $\lambda_1 \cdots \lambda_k $ on the left and right sides of the last equality. 

By Minkowski's formula \eqref{mixed1} applied to the left-hand side, we obtain the coefficient $$ k! \ \Tilde{V}_{d+1}(K_{1}, \ldots, K_{k}, B^{d+1},\dots,B^{d+1}) {d+1 \choose k}\lambda^{d+1-k}.$$

On the right-hand side, the coefficient of $\lambda_1 \cdots \lambda_k $ is 

$$ k! \ \Tilde{V}_{d}(K_{1}, \ldots, K_{k}, B^{d},\dots,B^{d}){d \choose k} \int_{-\lambda}^{\lambda}\sqrt{(\lambda^2 - z^2)}^{d-k}dz.$$

Notice that \begin{align*}\int_{-\lambda}^{\lambda}\sqrt{(\lambda^2 - z^2)}^{d-k}dz = \lambda^{d+1-k}\int_{-1}^{1}\sqrt{(1 - z^2)}^{d-k}dz = \lambda^{d+1-k} \frac{\sqrt{\pi} \Gamma(\frac{d-k}{2}+1)}{\Gamma(\frac{d-k+3}{2})} .
\end{align*}
Comparing the coefficients, we obtain
  \begin{align*}
k! \ \Tilde{V}_{d+1}(K_{1}, \ldots, K_{k}, B^{d+1},\dots,B^{d+1}) {d+1 \choose k}\lambda^{d+1-k} 
\\
=  k! \ \Tilde{V}_{d}(K_{1}, \ldots, K_{k}, B^{d},\dots,B^{d})  {d \choose k} \lambda^{d+1-k}\int_{-1}^{1}\sqrt{(1 - z^2)}^{d-k}dz.
\end{align*}
Taking into account the value of $\kappa_k:=\pi^{k/2}/\Gamma(\frac k 2 +1)$ and the last equality, we have
\begin{align*}
    {d+1 \choose k}\Tilde{V}_{d+1}(K_{1}, \ldots, K_{k}, B^{d+1},\dots,B^{d+1})\kappa_{d-k}
    \\
    = {d \choose k}  \Tilde{V}_{d}(K_{1}, \ldots, K_{k}, B^{d},\dots,B^{d})\kappa_{d+1-k},
\end{align*}
which completes the proof.
\section {Proof of Theorem \ref{mainmix}} \label{mainmixproof}
We divide the proof of the theorem into two cases.

\textbf{Case 1.} $\dim K_i < \infty$ for all $i= 1, \ldots, k$. In this case, taking into account Remarks \ref{finited}, \ref{norm}, the statement of Theorem \ref{mainmix} can be rewritten in the following form.

%



\begin{st} \label{stat}
If $ K_1, \ldots, K_k \subset \R^d, \ k = 1, \ldots, d$, then
\begin{align*}
    	\Tilde{V}_d(K_{1}, \ldots, K_{k}, B^d,\dots,B^d)= c_{k,d}\mathbf{E}\, \Tilde{V}_k(AK_{1}, \ldots, AK_{k}),
\end{align*}
where $c_{k,d} = \frac{\kappa_{d-k}(2\pi)^{k/2}}{ k!{d \choose k}\kappa_k }$, 
$AK_i:=\{Ax:x\in K_i\}\subset \R^k$ and $A$ is the standard Gaussian matrix of size $k\times d$.
\end{st}

\begin{proof}[Proof of Statement \ref{stat}]
Let us look at $\Vol_d \left(\sum_{i=1}^k \alpha_i K_i + (d-1) \lambda B^d \right)$. 

By  Minkowski's theorem \eqref{mixed1},
\begin{align}\label{vol1}
    \Vol_d \left(\sum_{i=1}^k \alpha_i K_i + (d-1) \lambda B^d \right) =\sum_{i_1=1}^{d+k-1} \cdots \sum_{i_d=1}^{d+k-1} \lambda_{i_1} \ldots \lambda_{i_d}  \Tilde{V}_d(K_{i_1}, \ldots, K_{i_d}),
\end{align}
where $ \lambda_{i_j} \in \{  \alpha_1, \ldots,  \alpha_k, \lambda\}, \ K_{i_j} \in \{ K_1, \ldots, K_k, B^d\}$.
Therefore, in the sum on the right-hand side, the coefficient of $\lambda^{d-k}$ is a polynomial in $\alpha_1, \ldots, \alpha_k$,
and the coefficient of $\lambda^{d-k} \alpha_1 \cdots \alpha_k$ is equal to $$k!{d \choose k}\Tilde{V}_d(K_{1}, \ldots, K_{k}, B^d,\dots,B^d). $$ 

On the other hand, considering $T:= \sum_{i=1}^k \alpha_i K_i$, by  Minkowski's theorem \eqref{mixed1}, we have
\begin{align}\label{vol2}
    \Vol_d (T + (d-1) \lambda B^d ) =\sum_{i_1=1}^d \cdots \sum_{i_d=1}^d \lambda_{i_1} \ldots \lambda_{i_d}  \Tilde{V}_d(K_{i_1}, \ldots, K_{i_d}),
\end{align}
where $ \lambda_{i_j} \in \{ 1, \lambda\}, \ K_{i_j} \in \{ T, B^d\}$.
In this case, since the mixed volumes are invariant with respect to permutations of the arguments, the coefficient of $\lambda^{d-k}$ will be equal to  $${d \choose k} \Tilde{V}_d( \underbrace{T,\dots,T}_{k\;\text{times}},B^d,\dots,B^d). $$

Further,
\begin{align*}
    \Tilde{V}_d( \underbrace{T,\dots,T}_{k\;\text{times}},B^d,\dots,B^d) \stackrel{\eqref{mixed}}{=} \frac{\kappa_{d-k}}{{d \choose k}}	V_k(T) = \frac{\kappa_{d-k}}{{d \choose k}} V_k \left(\sum_{i=1}^k \alpha_i K_i \right) 
    \\
    \stackrel{\eqref{2042}}{=} \frac{\kappa_{d-k}}{{d \choose k}}  \frac{(2\pi)^{k/2}}{k!\kappa_k} \mathbf{E} \,\Vol_k \left(A\left(\sum_{i=1}^k \alpha_i K_i \right)\right)
    \\
    =  \frac{\kappa_{d-k}}{{d \choose k}}  \frac{(2\pi)^{k/2}}{k!\kappa_k} \mathbf{E}\,\Vol_k \left(\sum_{i=1}^k \alpha_i A K_i \right).
\end{align*}

Applying again Minkowski's theorem \eqref{mixed1}, we conclude that $\mathbf{E}\,\Vol_k \left(\sum_{i=1}^k \alpha_i A K_i \right)$ is a homogeneous polynomial of degree $k $ in $\alpha_1,\ldots, \alpha_k$ with coefficient of $ \alpha_1 \cdots \alpha_k$ equal to $$k!\ \mathbf{E}\,\Tilde{V}_k(AK_{1}, \ldots, AK_{k}).$$ 

Thus, on the right side of \eqref{vol2} the coefficient of $\lambda^{d-k} \alpha_1 \cdots \alpha_k $ equals $$\kappa_{d-k} \frac{(2\pi)^{k/2}}{\kappa_k} \mathbf{E}\, \Tilde{V}_k(AK_{1}, \ldots, AK_{k}).$$

So, the left-hand sides of  relations \eqref{vol1} and \eqref{vol2} are the same. Hence, the coefficients of $\lambda^{d-k} \alpha_1 \cdots \alpha_k$ are the same on the right-hand sides:
\begin{align*}
    k!{d \choose k}\Tilde{V}_d(K_{1}, \ldots, K_{k}, B^d,\dots,B^d) = \kappa_{d-k} \frac{(2\pi)^{k/2}}{\kappa_k} \mathbf{E}\, \Tilde{V}_k(AK_{1}, \ldots, AK_{k}),
    \\
    \Tilde{V}_d(K_{1}, \ldots, K_{k}, B^d,\dots,B^d) =  \frac{\kappa_{d-k}(2\pi)^{k/2}}{ k!{d \choose k}\kappa_k } \mathbf{E}\, \Tilde{V}_k(AK_{1}, \ldots, AK_{k}).
\end{align*}
This completes the proof of Statement \ref{stat}.

\end{proof}
\textbf{Case 2.} $\dim K_i = \infty$ for at least one index $i=1, \ldots, k$.

According to Statement \ref{st2}, the $GB$-property of the compact sets $K_i$ implies that the processes $\langle \theta, x \rangle, \theta \in K_i,$ have a natural modification.


Next, we reduce Case 2 to the finite-dimensional one (Statement \ref{stat}).
Let $K_{1,1}\subset K_{1,2} \subset \ldots \subset K_1, \ K_{2,1}\subset K_{2,2} \subset \ldots \subset K_2, \ldots, \ K_{k,1}\subset K_{k,2} \subset \ldots \subset K_k$. Here $K_{i,j}$ are finite-dimensional convex compact sets, and $\cup_{j=1}^{\infty} K_{i,j} $ is dense in $K_i$. Then by definition \eqref{commonmix} and by properties \ref{monoton}, \ref{contin} of mixed volumes, 
we get 
\begin{align*}
    \Tilde{V}(K_{1}, \ldots, K_{k}) = \lim_{j \to \infty} \Tilde{V}(K_{1,j}, \ldots, K_{k,j}).
\end{align*}
Now we formulate the lemma proved by Tsirelson in \cite{Ts1985}.
\begin{lm}\label{lemmaT}
Let $\xi(t,\omega)$ be a natural modification of some random process,
$\omega \in \Omega$, $t \in T$, and let $S \subset T$ be dense in $T$ in the following sense: for any $t \in T$ there are $s_n \in S$, $n = 1, 2, \ldots,$ such that $\xi(s_n,\omega) \rightarrow \xi(t,\omega)$ as $n\rightarrow \infty$ for almost
all $\omega$ (the corresponding set of probability 1, generally speaking,  
depends on $t$). Then there exists a set $\Omega_0 \subset \Omega$ of probability 1 with the following property: for any
$t \in T$ there are $s'_n \in S$, $n = 1, 2, \ldots,$ such that $\xi(s'_n,\omega) \rightarrow \xi(t ,\omega)$ as $n\rightarrow \infty$ for all $\omega \in \Omega_0$.
\end{lm}

\begin{cor}\label{spec}
Let $K \subset E_0$ be a convex compact $GB$-set. If $K_0 \subset K$ is dense in~$K$, then for almost all $(x_1, \ldots, x_k)$ the set $ \spec(x_1, \ldots, x_k | K_0)$ is dense in the set $ \spec(x_1, \ldots, x_k | K)$.
\end{cor}
The above corollary is stated in \cite{Ts1985} without proof. For the reader’s convenience, we 
prove it here. 
\begin{proof}[Proof of Corollary \ref{spec}]
Since $K_0$ is dense in $K$ in the usual sense, for any $\theta \in K$ there exist $s_n \in K_0$, $n = 1, 2, \ldots,$ such that $\norm[s_n - \theta] \rightarrow 0 $ as $n \rightarrow \infty$. Moreover, by property \eqref{isometr},
\begin{align*}
    \int_E \langle s_n - \theta,x \rangle^2 \gamma(dx) = \norm[s_n - \theta]^2  
     \rightarrow 0
\end{align*}
as $n \rightarrow \infty$.
 Therefore, the sequence $\langle s_n - \theta,x \rangle^2 $ converges to $0$ in probability. Then there is a subsequence (we will also denote it by $s_n$) such that $\langle s_n - \theta,x \rangle^2 $ converges to $0$ almost surely (the corresponding set of probability $1$ depends on $\theta$).


Since $K \subset E_0$ is a convex compact $GB$-set, the Gaussian process $\langle \theta, x \rangle$ has a natural modification by 
Statement \ref{st2}. Hence, by Lemma \ref{lemmaT}, for any $\theta \in K$ there exist $s'_n \in K_0$ such that 
\begin{align} \label{coord}
\langle s'_n, x \rangle \rightarrow \langle \theta, x \rangle
\end{align}
as $n \rightarrow \infty$ for almost all $x$, and the corresponding set of probability $1$ is common for all $\theta \in K$.

Then we can conclude that for almost all $(x_1, \ldots, x_k)$ the set
$$
\spec(x_1, \ldots, x_k | K_0) = \{ (\langle \theta,x_1 \rangle, \ldots, \langle \theta,x_k \rangle) : \theta \in K_0 \} \subset \R^ k
$$
 is dense in the set
$$
\spec(x_1, \ldots, x_k | K) = \{ (\langle \theta,x_1 \rangle, \ldots, \langle \theta,x_k \rangle) : \theta \in K \} \subset \R^ k,
$$
since the argument above implies a coordinate-wise density \eqref{coord}, and the corresponding set of probability $1$ in $\R^k$ will also be common for all $\theta \in K$.

\end{proof}
Using Corollary \ref{spec}, we get that almost surely $\cup_{j=1}^{\infty}\spec(x_1, \ldots, x_k | K_{i,j})$ is  dense in
$\spec(x_1, \ldots, x_k | K_i)$.

Then we use Statement \ref{stat} for the finite-dimensional $K_{i,j}$:
\begin{align*}
   \Tilde{V}(K_{1}, \ldots, K_{k}) 
   =  \lim_{j \to \infty} \Tilde{V}(K_{1,j}, \ldots, K_{k,j})
  = \lim_{j \to \infty} \frac{(2\pi)^{k/2}}{k!\kappa_k}  
  \\
  \times \int_{E} \ldots \int_{E} \Tilde{V}_k(\spec(x_1, \ldots, x_k | K_{1,j}), \ldots, \spec(x_1, \ldots, x_k | K_{k,j})) \gamma(dx_1) \ldots \gamma(dx_k) 
   \\
   = \frac{(2\pi)^{k/2}}{k!\kappa_k} 
   \\
   \times \int_{E} \ldots \int_{E} \lim_{j \to \infty} \Tilde{V}_k(\spec(x_1, \ldots, x_k | K_{1,j}), \ldots, \spec(x_1, \ldots, x_k | K_{k,j})) \gamma(dx_1) \ldots \gamma(dx_k) 
   \\
   = \frac{(2\pi)^{k/2}}{k!\kappa_k} 
   \\
   \times 
   \int_{E} \ldots \int_{E} \Tilde{V}_k(\spec(x_1, \ldots, x_k | K_1), \ldots, \spec(x_1, \ldots, x_k | K_k)) \gamma(dx_1) \ldots \gamma(dx_k).
\end{align*}
Here in the third equality we have used 
Lebesgue's dominated convergence theorem.
The last equality also follows from properties \ref{monoton}, \ref{contin} of mixed volumes and  Corollary \ref{spec}.

The proof of Theorem \ref{mainmix} is complete.




\section{Proof of Theorem \ref{thspW}} \label{example}

Since $K_1$ and $K_2$ are compact $GB$-sets, we can apply Theorem \ref{mainmix} with $k=2$:
\begin{gather*}
    \Tilde{V}(K_{1}, K_{2}) = \frac{2\pi}{2\kappa_2}  \ \mathbf{E} \
    \Tilde{V}_2(\spec_2 K_1, \spec_2 K_2).
\end{gather*}
Since $\kappa_2 = \pi$, we get
\begin{gather}
    \Tilde{V}(K_{1}, K_{2}) = \mathbf{E} \ 
    \Tilde{V}_2(\spec_2 K_1, \spec_2 K_2)\nonumber
    \\= \mathbf{E} \ \Tilde{V}_2 \left(\conv \left(\{X_{1}^{(2)}(t) \colon t\in [0,1]\} \right), \conv \left(\{X_{2}^{(2)}(t) \colon t\in [0,1]\}\right)\right) \label{exmain},
\end{gather}    
where $X_{1}^{(2)}(t), \ X_{2}^{(2)}(t) $ are independent standard two-dimensional Brownian motions, and $\conv(F)$ denotes a closed convex hull of the set $F$.

Therefore, our problem is reduced to finding the mean mixed area $\mathbf{E} \ \Tilde{V}_2$ of the
closed convex hulls of independent two-dimensional Brownian motions on $[0,1]$. 

Further, for calculation, we will use an analogue of the technique given in \cite{majumdar2010random}. The main tools of this  technique are the support function 
and the associated Cauchy's formulae.

Let $C$ be an arbitrary closed smooth convex curve in a plane. Let us represent the curve $C$ as 
$$ C = \{(x(s), y(s))\}, \ s \in C.$$
Now we recall the notion of the support function of the curve $C$.

For $\phi \in [0, 2\pi)$ the value of the \textit{support function} $M(\phi)$ of the curve $C$ is defined by
\begin{align*}
    M(\phi) = \max_{s \in C} \{ x(s) \cos \phi + y(s) \sin \phi\}.
\end{align*}


The Cauchy's formulae (see,  e.g., \cite[pp. 48-49]{majumdar2010random}) allow us to express the length $L$ of the curve $C$ and the area $A$ of the figure bounded by the curve $C$ in terms of the support function:
\begin{align}
    &L = \int_{0}^{2 \pi} M(\phi) d\phi, \label{perimetr}
    \\
    &A = \frac{1}{2}\int_{0}^{2 \pi} \left( (M(\phi))^2 - (M'(\phi))^2 \right)d\phi.\label{area}
\end{align}

In the case when the curve $C$ is random (for example, the boundary of a closed convex hull of a two-dimensional Brownian motion, this curve is almost surely smooth~\cite{el1983enveloppe}), $M(\phi)$ and $M'(\phi)$ are random variables.

Taking the expectation of both sides of relations \eqref{perimetr}, \eqref{area}, we get
\begin{align}
    &\mathbf{E}L = \int_{0}^{2 \pi} \mathbf{E}M(\phi) d\phi, \label{eperimetr}
    \\
    &\mathbf{E}A = \frac{1}{2}\int_{0}^{2 \pi} \left( \mathbf{E}(M(\phi))^2 - \mathbf{E}(M'(\phi))^2\right)d\phi.\label{earea}
\end{align}

Note that the distribution of the two-dimensional Brownian motion is invariant under rotations. Hence, the distribution of the support function $M(\phi)$ does not depend on $\phi$. Therefore, it is sufficient to consider $\phi = 0, M(\phi) \equiv M(0)$. Relations \eqref{eperimetr}, \eqref{earea} in this case can be written in the form
\begin{align}
    &\mathbf{E}L = 2 \pi \ \mathbf{E}M(0), \nonumber 
    \\
    &\mathbf{E}A = \pi \left( \mathbf{E}(M(0))^2 - \mathbf{E}(M'(0))^2\right).\label{isoearea}
\end{align}

The following expression (see, e.g., \cite[pp. 4-5]{MR0433364}) is an analogue 
of the Cauchy's formulae for
 computation 
of the mixed area of two convex compact sets $F_1, F_2$ in a plane with smooth boundary:

\begin{align*}
     \Tilde{V}_2(F_{1}, F_{2}) = \frac{1}{2}\int_{0}^{2 \pi} \left( M_1(\phi)M_2(\phi) - M'_1(\phi)M'_2(\phi) \right) d\phi,
\end{align*}
where $M_1$ and $ M_2$ are the support functions of the curves representing the boundaries of $F_1$ and $F_2$, respectively.

Similarly to \eqref{earea}, for random $F_1, F_2$ we get
\begin{align}
     \mathbf{E} \ \Tilde{V}_2(F_{1}, F_{2}) = \frac{1}{2}\int_{0}^{2 \pi} \left( \mathbf{E}(M_1(\phi)M_2(\phi)) - \mathbf{E}(M'_1(\phi)M'_2(\phi)) \right) d\phi.\label{mixarea}
\end{align}

Now consider $ \conv \left(\{X_{1}^{(2)}(t) \colon t\in [0,1]\} \right)$ and $\conv \left(\{X_{2}^{(2)}(t) \colon t\in [0,1]\} \right) $ as $F_1$ and $F_2$, respectively. By formula \eqref{mixarea} and the independence of $X_{1}^{(2)}(t)$ and $X_{2}^{(2)}(t) $, we have 
\begin{gather}
    \mathbf{E} \ \Tilde{V}_2 \left(\conv \left(\{X_{1}^{(2)}(t) \colon t\in [0,1]\} \right), \conv \left(\{X_{2}^{(2)}(t) \colon t\in [0,1]\}\right)\right) \nonumber
    \\
    = \frac{1}{2}\int_{0}^{2 \pi} \left( \mathbf{E}M_1(\phi)\mathbf{E}M_2(\phi) - \mathbf{E}M'_1(\phi)\mathbf{E}M'_2(\phi) \right) d \phi \nonumber
    \\
    = \frac{1}{2} \ 2 \pi \left( (\mathbf{E}M_1(0))^2 - (\mathbf{E}M'_1(0))^2 \right) 
    = \pi \left( (\mathbf{E}M_1(0))^2 - (\mathbf{E}M'_1(0))^2 \right) \label{exicase}.
\end{gather}   
Here the second equality follows from relation \eqref{isoearea} and the fact that $M_1$ and $M_2$ are identically distributed. 

Thus, it remains to calculate $\mathbf{E}M_1(0)$ and $\mathbf{E}M'_1(0)$, where $M_1$ is the support function of the boundary of the convex hull of the two-dimensional Brownian motion on~$[ 0,1]$.

Recall that
$$
\{X^{(2)}(t) \colon t\in [0,1]\} = \{(W_1(t),W_2(t))\colon t\in [0,1] \} ,
$$
where $W_1(t)$ and $ W_2(t)$ are independent standard one-dimensional Brownian motions.

We fix a direction $\phi$. For $t \in [0,1]$, consider projections on the direction $\phi$ and perpendicular to it:
\begin{align*}
     z_{\phi} (t) = W_1(t) \cos \phi + W_2(t) \sin \phi,
     \\
     h_{\phi} (t) = -W_1(t) \sin \phi + W_2(t) \cos \phi.
\end{align*}
Then $z_{\phi}$ and $ h_{\phi}$ are independent standard one-dimensional Brownian motions on~$[0,1]$. 

Consequently, the support function
\begin{align*}
     M_1(\phi) = \max_{t \in [0,1]}z_{\phi} (t)
\end{align*}
is the maximum of the one-dimensional Brownian motion $z_{\phi}$ on~$[0,1]$.

Let $t^* \in [0,1]$ be the time when 
this maximum is attained. Then
\begin{align}\label{opora}
    M_1(\phi) =z_{\phi} (t^*) =  W_1(t^*) \cos \phi + W_2(t^*) \sin \phi.
\end{align}

Differentiating \eqref{opora} with respect to $\phi$, we have
\begin{align*}
     M'_1(\phi) = -W_1(t^*) \sin \phi+ W_2(t^*) \cos \phi = h_{\phi}(t^*).
\end{align*}
In other words, $M_1(\phi)$ is the maximum of the first Brownian motion $z_{\phi}$, and $M'_1(\phi)$ corresponds to the value of the second, independent Brownian motion $h_{\phi}$ at time $t^*$ when the first one attains its maximum.

In particular, for $\phi = 0$ we obtain $z_0(t) = W_1(t)$, $h_0(t) = W_2(t)$, and
\begin{align*}
   &M_1(0) =\max_{t \in [0,1]}W_1 (t),
   \\
   &M'_1(0) = W_2(t^*).
\end{align*}

  The cumulative distribution function of the maximum of one-dimensional Brownian motion on $[0,1]$ is well known (see, e.g., \cite{MR0270403}), namely
\begin{align*}
  F(m)=\mathbf{P}\left(\max_{t \in [0,1]}W_1 (t)\le m \right)={\rm erf}\left(\frac{m}{\sqrt{2}}\right),     \end{align*}
 where
${\rm erf}(z)= \frac{2}{\sqrt{\pi}}\,\int_0^z e^{-u^2}\, du$. The first moment of this distribution is easily calculated:
\begin{align} \label{fmom}
     \mathbf{E}M_1(0) = \mathbf{E}\max_{t \in [0,1]}W_1 (t) =
\sqrt{\frac{2}{\pi}}.
\end{align}

Let us show that
\begin{align} \label{fmomder}
    \mathbf{E}M'_1(0) = \mathbf{E}W_2(t^*) =0.
\end{align}

Indeed, since $t^*$ and $W_2(t)$ are independent, we see that
\begin{align}\label{der}
     \mathbf{E}M'_1(0) = \mathbf{E}W_2(t^*) =
\int_{0}^1 \int_{-\infty}^{\infty} x p_1(x,t) dx \ p_2(t) dt.
\end{align}
Here $p_1$ denotes the density of the normal distribution $N(0, t)$ under the condition that $t^*=t$, and $p_2$ denotes the density of the random variable $t^*$ (an explicit formula for $p_2 $ can be found in \cite{majumdar2010random}). Since for a fixed $t \in [0,1]$ the inner integral in \eqref{der} equals $0$, we have $\mathbf{E}M'_1(0) = 0$.

Combining \eqref{exmain}, \eqref{exicase}, \eqref{fmom} and \eqref{fmomder}, we get
\begin{align*}
     \Tilde{V}(K_{1}, K_{2}) = \pi \left( (\mathbf{E}M_1(0))^2 - (\mathbf{E}M'_1(0))^ 2\right) = 2.
\end{align*}
\section{Acknowledgments}


The author is grateful to Dmitry Zaporozhets for helpful discussions and valuable comments.

\bibliographystyle{plain}
\bibliography{plan}

\end{document}